\documentclass[a4paper,11pt]{article}

\usepackage[defaultlines=5,all]{nowidow}		
\usepackage{indentfirst}
\usepackage[T1]{fontenc}
\usepackage[utf8]{inputenc}
\usepackage{authblk}
\usepackage{diagbox}
\usepackage{amsthm,amsmath}
\usepackage{tikz}
\usepackage{paralist}
\usepackage[shortlabels]{enumitem}
\usepackage{subcaption}
\usepackage{hyperref}		
\usepackage{stfloats}
\usepackage{verbatim}
\hypersetup{
	colorlinks=true,
	citecolor = blue,
	linkcolor=blue,
	filecolor=magenta,
	urlcolor=cyan,
}

\newtheorem{theorem}{Theorem}
\newtheorem{lemma}{Lemma}
\newtheorem{claim}{Claim}

\title{{\bf Critical ($P_5$,bull)-free graphs}}

\author[,a,b]{Shenwei Huang\thanks{Email: shenweihuang@nankai.edu.cn. Supported by Natural Science Foundation of Tianjin (20JCYBJC01190), and the Fundamental Research Funds for the Central Universities, Nankai University.}}
\author[a]{Jiawei Li}
\author[a]{Wen Xia}
\affil[a]{College of Computer Science, Nankai University, Tianjin 300071, China}
\affil[b]{Tianjin Key Laboratory of Network and Data Security Technology, Nankai University, Tianjin 300071, China}

\date{}

\begin{document}
	
	\maketitle
	
	\begin{abstract}
		Given two graphs $H_1$ and $H_2$, a graph is $(H_1,H_2)$-free if it contains no induced subgraph isomorphic to $H_1$ or $H_2$. Let $P_t$ and $C_t$ be the path and the cycle on $t$ vertices, respectively. A bull is the graph obtained from a triangle with two disjoint pendant edges. In this paper, we show that there are finitely many 5-vertex-critical ($P_5$,bull)-free graphs.
		
		{\bf Keywords.} coloring; critical graphs; forbidden induced subgraphs; strong perfect graph theorem; polynomial-time algorithms.
	\end{abstract}
	
	\section{Introduction}
	All graphs in this paper are finite and simple. We say that a graph $G$ {\em contains} a graph $H$ if $ H $ is isomorphic to an induced subgraph of $G$. A graph $G$ is {\em H-free} if it does not contain $H$. For a family of graphs $\mathcal{H}$, $G$ is {\em $\mathcal{H}$-free} if $G$ is $H$-free for every $H\in \mathcal{H}$. When $\mathcal{H}$ consists of two graphs, we write $(H_1,H_2)$-free instead of $\{H_1,H_2\}$-free.
	
	A $k$-{\em coloring} of a graph $G$ is a function $\phi:V(G)\rightarrow\{1,...,k\}$ such that $\phi(u)\neq\phi(v)$ whenever $u$ and $v$ are adjacent in $G$. Equivalently, a $k$-coloring of $G$ is a partition of $V(G)$ into $k$ independent sets. We call a graph $k$-{\em colorable} if it admits a $k$-coloring. The {\em chromatic number} of $G$, denoted by $\chi(G)$, is the minimum number $k$ for which $G$ is $k$-colorable. The {\em clique number} of $G$, denoted by $\omega(G)$, is the size of a largest clique in $G$.
	
	A graph $G$ is said to be $k$-{\em chromatic} if $\chi(G)=k$. We say that $G$ is {\em critical} if $\chi(H)<\chi(G)$ for every proper subgraph $H$ of $G$. A $k$-{\em critical} graph is one that is $k$-chromatic and critical. An easy consequence of the definition is that every critical graph is connected. Critical graphs were first investigated by Dirac \cite{Di51,Di52,Di52i} in 1951, and then by Lattanzio and Jensen \cite{L02,J02} among others, and by Goedgebeur \cite{GS18} in recent years.
	
	Vertex-criticality is a weaker notion. Suppose that $G$ is a graph. Then $G$ is said to be $k$-{\em vertex-critical} if $G$ has chromatic number $k$ and removing any vertex from $G$ results in a graph that is $(k-1)$-colorable. For a set $\mathcal{H}$ of graphs, we say that $ G $ is {\em k-vertex-critical $\mathcal{H}$-free} if it is $k$-vertex-critical and $\mathcal{H}$-free. The following problem arouses our interest.
	
	{\noindent} \textbf{The finiteness problem.} Given a set $\mathcal{H}$ of graphs and an integer $k\ge 1$, are there only finitely many $ k $-vertex-critical $\mathcal{H}$-free graphs? 
	
	This problem is meaningful because the finiteness of the set has a fundamental algorithmic implication.
	
	\begin{theorem}[Folklore]\label{Folklore}
		If the set of all $k$-vertex-critical $\mathcal{H}$-free graphs is finite, then there is a polynomial-time algorithm to determine whether an $\mathcal{H}$-free graph is $(k-1)$-colorable.  \qed
	\end{theorem}

	Let $K_n$ be the complete graph on $n$ vertices. Let $ P_t $ and $ C_t $ denote the path and the cycle on $t$ vertices, respectively. The {\em complement} of $G$ is denoted by $\overline{G}$. For $s,r\ge 1$, let $K_{r,s}$ be the complete bipartite graph with one part of size $r$ and the other part of size $s$. A class of graphs that has been extensively studied recently is the class of $P_t$-free graphs. In \cite{BHS09}, it was shown that there are finite many 4-vertex-critical $P_5$-free graphs. This result was later generalized to $P_6$-free graphs \cite{CGSZ16}. In the same paper, an infinite family of 4-vertex-critical $P_7$-free graphs was constructed. Moreover, for every $k\ge 5$, an infinite family of $k$-vertex-critical $P_5$-free graphs was constructed in \cite{HMRSV15}. This implies that the finiteness of $k$-vertex-critical $P_t$-free graphs for $t\ge 1$ and $k\ge 4$ has been completely solved	 by researchers. We summarize the results in the following table.
	
	\begin{table}[!ht]
		\centering 
		\caption{The finiteness of $k$-vertex-critical $P_t$-free graphs.}
		\renewcommand\arraystretch{1.5}  
		\setlength{\tabcolsep}{4mm}{}
		\begin{tabular}{|c|p{1.6cm}<{\centering}|p{1.9cm}<{\centering}|p{1.6cm}<{\centering}|p{1.7cm}<{\centering}|}
		\hline
		\diagbox{$k$}{$t$} & $\le4$ & 5 & 6 & $\ge7$\\
		\hline
		4 & finite & finite \cite{BHS09}& finite \cite{CGSZ16}& infinite \cite{CGSZ16}\\
		\hline
		$\ge 5$ & finite & infinite \cite{HMRSV15}& infinite & infinite \\  
		\hline 
	\end{tabular}
	\end{table}

	Because there are infinitely many 5-vertex-critical $P_5$-free graphs, many researchers have investigated the finiteness problem of $k$-vertex-critical $(P_5,H)$-free graphs. Our research is mainly motivated by the following dichotomy result.
	
	\begin{theorem}[\cite{CGHS21}]
		Let $H$ be a graph of order 4 and $k\ge 5$ be a fixed integer. Then there are infinitely many k-vertex-critical $(P_5,H)$-free graphs if and only if $H$ is $2P_2$ or $P_1+K_3$.
	\end{theorem}

	This theorem completely solves the finiteness problem of $k$-vertex-critical $(P_5,H)$-free graphs for graphs of order 4. In \cite{CGHS21}, the authors also posed the natural question of which five-vertex graphs $H$ lead to finitely many $k$-vertex-critical $(P_5,H)$-free graphs. It is known that there are exactly 13 5-vertex-critical $(P_5,C_5)$-free graphs \cite{HMRSV15}, and that there are finitely many 5-vertex-critical ($P_5$,banner)-free graphs \cite{CHLS19,HLS19}, and finitely many $k$-vertex-critical $(P_5,\overline{P_5})$-free graphs for every fixed $k$ \cite{DHHMMP17}. In \cite{CGS},  Cai, Goedgebeur and Huang show that there are finitely many $k$-vertex-critical ($P_5$,gem)-free graphs and finitely many $k$-vertex-critical ($P_5,\overline{P_3+P_2}$)-free graphs. Hell and Huang proved that there are finitely many $k$-vertex-critical $(P_6,C_4)$-free graphs \cite{HH17}. This was later generalized to $(P_5,K_{r,s})$-free graphs in the context of $H$-coloring \cite{KP17}. This gives an affirmative answer for $H=K_{2,3}$.
	
	\noindent {\bf Our contributions.} We continue to study the finiteness of vertex-critical $(P_5,H)$-free graphs when $H$ has order 5. The {\em bull} graph (see \autoref{bull}) is the graph obtained from a triangle with two disjoint pendant edges. In this paper, we prove that there are only finitely many 5-vertex-critical ($P_5$,bull)-free graphs.
	
		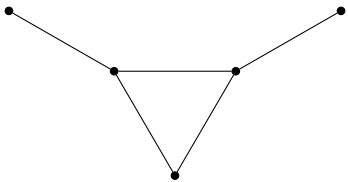
\begin{figure}[h]
		\centering
		\begin{tikzpicture}[scale=0.8]
			\tikzstyle{vertex}=[draw, circle, fill=black!100, minimum width=1pt,inner sep=1pt]
			
			\node[vertex](v1) at (-1,1.73) {};
			\node[vertex](v2) at (0,0) {};
			\node[vertex](v3) at (1,1.73) {};
			\node[vertex](v4) at (-2.73,2.73){};
			\node[vertex](v5) at (2.73,2.73) {};
			\draw (v1)--(v2)--(v3)--(v1) (v1)--(v4) (v3)--(v5);	
		\end{tikzpicture}
		\caption{The bull graph.}
		\label{bull}
	\end{figure}
	
	To prove the result on bull-free graphs, we performed a careful structural analysis combined with the pigeonhole principle based on the properties of 5-vertex-critical graphs.
	
	The remainder of the paper is organized as follows. We present some preliminaries in Section \ref{Preliminarlies} and give structural properties around an induced $C_5$ in a ($P_5$,bull)-free graph in Section \ref{structure}. We then show that there are finitely many 5-vertex-critical ($P_5$,bull)-free graphs in Section \ref{bull-free}.
	
	\section{Preliminaries}\label{Preliminarlies}
	
	For general graph theory notation we follow \cite{BM08}. For $k\ge 4$, an induced  cycle of length $k$ is called a {\em $k$-hole}. A $k$-hole is an {\em odd hole} (respectively {\em even hole}) if $k$ is odd (respectively even). A {\em $k$-antihole} is the complement of a $k$-hole. Odd and even antiholes are defined analogously.
	
	Let $G=(V,E)$ be a graph. For $S\subseteq V$ and $u\in V\setminus S$, let $d(u,S)={min}_{v\in S}d(u,v)$, where $d(u,v)$ denotes the length of the shortest path from $u$ to $v$. If $uv\in E$, we say that $u$ and $v$ are {\em neighbors} or {\em adjacent}, otherwise $u$ and $v$ are {\em nonneighbors} or  {\em nonadjacent}. The {\em neighborhood} of a vertex $v$, denoted by $N_G(v)$, is the set of neighbors of $v$. For a set $X\subseteq V$, let $N_G(X)=\cup_{v\in X}N_G(v)\setminus X$. We shall omit the subscript whenever the context is clear. For $x\in V$ and $S\subseteq V$, we denote by $N_S(x)$ the set of neighbors of $x$ that are in $S$, i.e., $N_S(x)=N_G(x)\cap S$. For two sets $X,S\subseteq V(G)$, let $N_S(X)=\cup_{v\in X}N_S(v)\setminus X$. For $X,Y\subseteq V$, we say that $X$ is {\em complete} (resp. {\em anticomplete}) to $Y$ if every vertex in $X$ is adjacent (resp. nonadjacent) to every vertex in $Y$. If $X=\{x\}$, we write ``$x$ is complete (resp. anticomplete) to $Y$'' instead of ``$\{x\}$ is complete (resp. anticomplete) to $Y$''. If a vertex $v$ is neither complete nor anticomplete to a set $S$, we say that $v$ is {\em mixed} on $S$. For a vertex $v\in V$ and an edge $xy\in E$, if $v$ is mixed on $\{x,y\}$, we say that $v$ is {\em mixed} on $xy$. For a set $H\subseteq V$, if no vertex in $V-H$ is mixed on $H$, we say that $H$ is a {\em homogeneous set}, otherwise $H$ is a {\em nonhomogeneous set}. A vertex subset $S\subseteq V$ is {\em independent} if no two vertices in $S$ are adjacent. A {\em clique} is the complement of an independent set. Two nonadjacent vertices $u$ and $v$ are said to be {\em comparable} if $N(v)\subseteq N(u)$ or $N(u)\subseteq N(v)$. A vertex subset $K\subseteq V$ is a {\em clique cutset} if $G-K$ has more connected components than $G$ and $K$ is a clique. For an induced subgraph $A$ of $G$, we write $G-A$ instead of $G-V(A)$. For $S\subseteq V$, the subgraph \emph{induced} by $S$ is denoted by $G[S]$. For $S\subseteq V$ and an induced subgraph $A$ of $G$, we may write $S$ instead of $G[S]$ and $A$ instead of $V(A)$ for the convenience of writing whenever the context is clear.
	
	We proceed with a few useful results that will be needed later. The first one is well-known in the study of $k$-vertex-critical graphs.
	
	\begin{lemma}[Folklore]\label{lem:xy}
		A $k$-vertex-critical graph contains no clique cutsets.
	\end{lemma}

	Another folklore property of vertex-critical graphs is that such graphs contain no comparable vertices. In \cite{CGHS21}, a generalization of this property was presented.
	
	\begin{lemma}[\cite{CGHS21}]\label{lem:XY}
		Let $G$ be a $k$-vertex-critical graph. Then $G$ has no two nonempty disjoint subsets $X$ and $Y$ of $V(G)$ that satisfy all the following conditions.
		\begin{itemize}
			\item $X$ and $Y$ are anticomplete to each other.
			\item $\chi(G[X])\le\chi(G[Y])$.
			\item Y is complete to $N(X)$.
		\end{itemize}
	\end{lemma}
	
	A property on bipartite graphs is shown as follows.
	
	\begin{lemma}[\cite{F93}]\label{2K2}
		Let $G$ be a connected bipartite graph. If $G$ contains a $2K_2$, then $G$ must contain a $P_5$.
	\end{lemma}

	As we mentioned earlier, there are finitely many 4-vertex-critical $P_5$-free graphs.

	\begin{theorem}[\cite{BHS09,MM12}]\label{thm:finite4Critical}
		If $G=(V,E)$ is a 4-vertex-critical $P_5$-free graph, then $|V|\le 13$.
	\end{theorem}

	A graph $G$ is {\em perfect} if $\chi(H)=\omega(H)$ for every induced subgraph $H$ of $G$. Another result we use is the well-known Strong Perfect Graph Theorem.
	
	\begin{theorem}[The Strong Perfect Graph Theorem\cite{CRST06}]\label{thm:SPGT}
		A graph is perfect if and only if it contains no odd holes or odd antiholes.
	\end{theorem}

	Moreover, we prove a property about homogeneous sets, which will be used frequently in the proof of our results.

\begin{lemma}\label{lem:homogeneous}
	Let $G$ be a 5-vertex-critical $P_5$-free graph and $S$ be a homogeneous set of $V(G)$. For each component $A$ of $G[S]$, 
	
	\begin{enumerate}[(i)]
		\item if $\chi(A)=1$, then $A$ is a $K_1$;
		\item if $\chi(A)=2$, then $A$ is a $K_2$;
		\item if $\chi(A)=3$, then $A$ is a $K_3$ or a $C_5$.
	\end{enumerate}
\end{lemma}

\begin{proof}
	(i) is clearly true. Moreover, since $V(A)\subseteq S$, $V(A)$ is also a homogeneous set. Next we prove (ii) and (iii).
	
	(ii)Since $\chi(A)=2$, let $\{x,y\}\subseteq V(A)$ induce a $K_2$. Suppose that there is another vertex $z$ in $A$. Because $G$ is 5-vertex-critical, $G-z$ is 4-colorable. Since $\chi(A)=2$, let $\{V_1,V_2,V_3,V_4\}$ be a 4-coloring of $G-z$ where $V(A)\setminus\{z\}\subseteq V_1\cup V_2$. Since $A$ is homogeneous, $\{V_1\cup \{z\},V_2,V_3,V_4\}$ or $\{V_1,V_2\cup \{z\},V_3,V_4\}$ is a 4-coloring of $G$, a contradiction. Thus $A$ is a $K_2$.
	
	(iii)We first show that $G$ must contain a $K_3$ or a $C_5$. If $A$ is $K_3$-free, then $\omega(A)<\chi(A)=3$ and so $A$ is imperfect. Since $A$ is $P_5$-free, $A$ must contain a $C_5$ by \autoref{thm:SPGT}. Thus $A$ contains either a $K_3$ or a $C_5$. 
	
	If $A$ contains a $K_3$ induced by $\{x,y,z\}$, suppose that there is another vertex $s$ in $A$. Because $G$ is 5-vertex-critical, $G-s$ is 4-colorable. Since $\chi(A)=3$, let $\{V_1,V_2,V_3,V_4\}$ be a 4-coloring of $G-s$ where $V(A)\setminus\{s\}\subseteq V_1\cup V_2\cup V_3$. Since $A$ is homogeneous, $\{V_1\cup \{s\},V_2,V_3,V_4\}$, $\{V_1,V_2\cup \{s\},V_3,V_4\}$ or $\{V_1,V_2,V_3\cup \{s\},V_4\}$ is a 4-coloring of $G$, a contradiction. Thus $A$ is a $K_3$. Similarly, $A$ is a $C_5$ if $A$ contains a $C_5$.
\end{proof}
	
	\section{Structure around a 5-hole}\label{structure}
	
	Let $G=(V,E)$ be a graph and $H$ be an induced subgraph of $G$. We partition $V\setminus V(H)$ into subsets with respect to $H$ as follows: for any $X\subseteq V(H)$, we denote by $S(X)$ the set of vertices in $V\setminus V(H)$ that have $X$ as their neighborhood among $V(H)$, i.e.,
	$$S(X)=\{v\in V\setminus V(H): N_{V(H)}(v)=X\}.$$
	\noindent For $0\le m\le|V(H)|$, we denote by $S_m$ the set of vertices in $V\setminus V(H)$ that have exactly $m$ neighbors in $V(H)$. Note that $S_m=\cup_{X\subseteq V(H):|X|=m}S(X)$.
	
	Let $G$ be a ($P_5$,bull)-free graph and $C=v_1,v_2,v_3,v_4,v_5$ be an induced $C_5$ in $G$. We partition $V\setminus C$ with respect to $C$ as above. All subscripts below are modulo five. Clearly, $S_1=\emptyset$ and so $V(G)=V(C)\cup S_0\cup S_2\cup S_3\cup S_4\cup S_5$. Since $G$ is ($P_5$,bull)-free, it is easy to verify that $S(v_i,v_{i+1})=S(v_{i-2},v_i,v_{i+2})=\emptyset$. So $S_2=\cup_{1\le i\le 5}S(v_{i-1},v_{i+1})$ and $S_3=\cup_{1\le i\le 5}S(v_{i-1},v_{i},v_{i+1})$. Note that $S_4=\cup_{1\le i\le 5}S(v_{i-2},v_{i-1},v_{i+1},v_{i+2})$. In the following, we write $S_2(i)$ for $S(v_{i-1},v_{i+1})$, $S_3(i)$ for $S(v_{i-1},v_{i},v_{i+1})$ and $S_4(i)$ for $S(v_{i-2},v_{i-1},v_{i+1},v_{i+2})$. We now prove a number of useful properties of $S(X)$ using the fact that $G$ is ($P_5$,bull)-free. All properties are proved for $i=1$ due to symmetry. In the following, if we say that $\{r,s,t,u,v\}$ induces a bull, it means that $r,v$ are two pendant vertices, $s$ is the neighbor of $r$, $u$ is the neighbor of $v$, and $stu$ is a triangle. If we say that $\{r,s,t,u,v\}$ induces a $P_5$, it means that any two consecutive vertices are adjacent.
	
	\begin{enumerate}[label=\bfseries (\arabic*)]
		
		\item {$S_2(i)$ is complete to $S_2(i+1)\cup S_3(i+1)$.}\label{S2(i)S2(i+1)}
		
		Let $x\in S_2(1)$ and $y\in S_2(2)\cup S_3(2)$. If $xy\notin E$, then $\{x,v_5,v_4,v_3,y\}$ induces a $P_5$.
		
		\item {$S_2(i)$ is anticomplete to $S_2(i+2)$.}\label{S2(i)S2(i+2)}
		
		Let $x\in S_2(1)$ and $y\in S_2(3)$. If $xy\in E$, then $\{v_3,v_2,y,x,v_5\}$ induces a bull.
		
		\item {$S_2(i)$ is anticomplete to $S_3(i+2)$.}\label{S2(i)S3(i+2)} 
		
		Let $x\in S_2(1)$ and $y\in S_3(3)$. If $xy\in E$, then $\{v_1,v_2,x,y,v_4\}$ induces a bull.
		
		\item {$S_2(i)$ is anticomplete to $S_4(i)$.}\label{S2(i)S4(i)} 
		
		Let $x\in S_2(1)$ and $y\in S_4(1)$. If $xy\in E$, then $\{v_1,v_2,x,y,v_4\}$ induces a bull. 
		
		\item {$S_2(i)\cup S_3(i)$ is complete to $S_4(i+2)$.}\label{S2(i)S4(i+2)} 
		
		Let $x\in S_2(1)\cup S_3(1)$ and $y\in S_4(3)$. If $xy\notin E$, then $\{v_3,v_4,y,v_5,x\}$ induces a bull. 
		
		\item {$S_2(i)$ is complete to $S_4(i+1)\cup S_5$.}\label{S2S5} 
		
		Let $x\in S_2(1)$ and $y\in S_4(2)\cup S_5$. If $xy\notin E$, then $\{v_3,y,v_1,v_5,x\}$ induces a bull. 
		
		\item {$S_3(i)$ is complete to $S_3(i+1)$.}\label{S3(i)S3(i+1)} 
		
		Let $x\in S_3(1)$ and $y\in S_3(2)$. If $xy\notin E$, then $\{x,v_5,v_4,v_3,y\}$ induces a $P_5$. 
		
	\end{enumerate}
	\section{The main result}\label{bull-free}
	
	Let $\mathcal{F}$ be the set of graphs shown in \autoref{fig:5vertexcritical}. It is easy to verify that all graphs in $\mathcal{F}$ are 5-vertex-critical.
	
	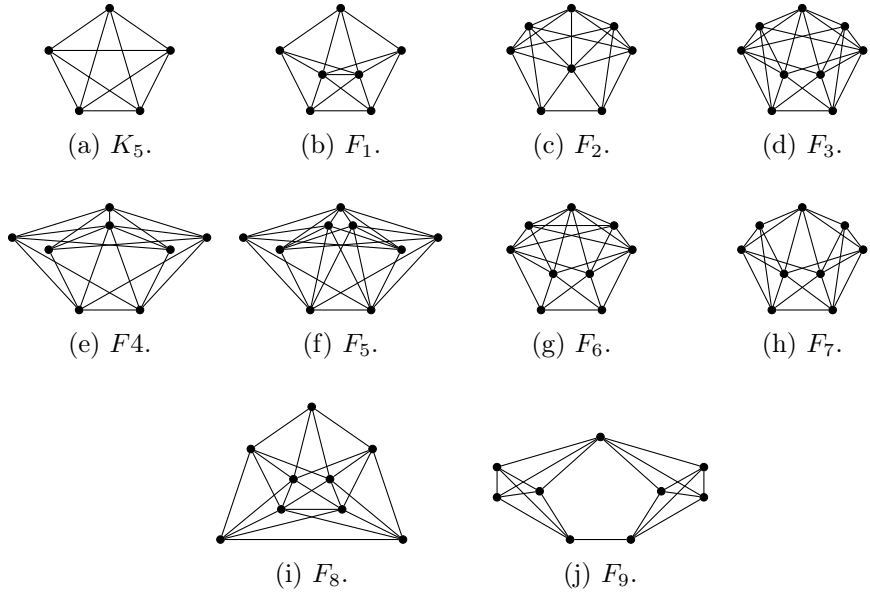
\begin{figure}[h]
		\centering
		\begin{subfigure}[t]{.24\textwidth}
			\centering
			\begin{tikzpicture}[scale=0.8]
				\tikzstyle{vertex}=[draw, circle, fill=black!100, minimum width=1pt,inner sep=1pt]
				
				\node[vertex] (v1) at (0.5,0) {};
				\node[vertex] (v2) at (1.5,0) {};
				\node[vertex] (v3) at (2,1) {};
				\node[vertex] (v4) at (1,1.7){};
				\node[vertex] (v5) at (0,1) {};
				\draw (v1)--(v2)--(v3)--(v4)--(v5)--(v1);
				
				\draw (v1)--(v3) (v1)--(v4) (v2)--(v4) (v2)--(v5) (v3)--(v5);
				
			\end{tikzpicture}
			\caption {$K_5$.}
			\label{K5}
		\end{subfigure}%
		\begin{subfigure}[t]{.24\textwidth}
			\centering
			\begin{tikzpicture}[scale=0.8]
				\tikzstyle{vertex}=[draw, circle, fill=black!100, minimum width=1pt,inner sep=1pt]
				
				\node[vertex] (v1) at (8,1.7) {};
				\node[vertex] (v2) at (7,1) {};
				\node[vertex] (v3) at (7.5,0) {};
				\node[vertex] (v4) at (8.5,0) {};
				\node[vertex] (v5) at (9,1) {};
				
				\draw (v1)--(v2)--(v3)--(v4)--(v5)--(v1);
				
				\node[vertex] (s51) at (7.7,0.6) {};
				\node[vertex] (s52) at (8.3,0.6) {};
				
				\draw (s51)--(s52);
				
				\foreach \i in {1,...,5}
				{
					\draw (s51)--(v\i);
					\draw (s52)--(v\i);
				}
				
			\end{tikzpicture}
			\caption{$F_1$.}
			\label{F1}
		\end{subfigure}%
		\begin{subfigure}[t]{.24\textwidth}
			\centering
			\begin{tikzpicture}[scale=0.8]
				\tikzstyle{vertex}=[draw, circle, fill=black!100, minimum width=1pt,inner sep=1pt]
				
				\node[vertex] (v1) at (8,1.7) {};
				\node[vertex] (v2) at (7,1) {};
				\node[vertex] (v3) at (7.5,0) {};
				\node[vertex] (v4) at (8.5,0) {};
				\node[vertex] (v5) at (9,1) {};
				
				\draw (v1)--(v2)--(v3)--(v4)--(v5)--(v1);
				
				\node[vertex] (s5) at (8,0.7) {};
				\node[vertex] (s41) at (7.3,1.4) {};
				\node[vertex] (s42) at (8.7,1.4) {};
				
				\draw (s5)--(s41) (s5)--(s42);;
				
				\foreach \i in {1,...,5}
				{
					\draw (s5)--(v\i);
				}
				\foreach \i in {1,2,3,5}
				{
					\draw (s41)--(v\i);
				}
				\foreach \i in {1,2,4,5}
				{
					\draw (s42)--(v\i);
				}
			\end{tikzpicture}
			\caption{$F_2$.}
			\label{F2}
		\end{subfigure}%
		\begin{subfigure}[t]{.24\textwidth}
			\centering
			\begin{tikzpicture}[scale=0.8]
				\tikzstyle{vertex}=[draw, circle, fill=black!100, minimum width=1pt,inner sep=1pt]
				
				\node[vertex] (v1) at (8,1.7) {};
				\node[vertex] (v2) at (7,1) {};
				\node[vertex] (v3) at (7.5,0) {};
				\node[vertex] (v4) at (8.5,0) {};
				\node[vertex] (v5) at (9,1) {};
				
				\draw (v1)--(v2)--(v3)--(v4)--(v5)--(v1);
				
				\node[vertex] (s51) at (7.7,0.6) {};
				\node[vertex] (s52) at (8.3,0.6) {};
				
				\foreach \i in {1,...,5}
				{
					\draw (s51)--(v\i);
					\draw (s52)--(v\i);
				}
				
				\node[vertex] (s41) at (7.3,1.4) {};
				\node[vertex] (s42) at (8.7,1.4) {};
				
				\foreach \i in {1,2,3,5}
				{
					\draw (s41)--(v\i);
				}
				\foreach \i in {1,2,4,5}
				{
					\draw (s42)--(v\i);
				}
				\draw (s41)--(s51) (s42)--(s52);
				
			\end{tikzpicture}
			\caption{$F_3$.}
			\label{F3}
		\end{subfigure}%
		\vspace{5mm}
		\begin{subfigure}[t]{.24\textwidth}
			\centering
			\begin{tikzpicture}[scale=0.8]
				\tikzstyle{vertex}=[draw, circle, fill=black!100, minimum width=1pt,inner sep=1pt]
				
				\node[vertex] (v1) at (8,1.7) {};
				\node[vertex] (v2) at (7,1) {};
				\node[vertex] (v3) at (7.5,0) {};
				\node[vertex] (v4) at (8.5,0) {};
				\node[vertex] (v5) at (9,1) {};
				
				\draw (v1)--(v2)--(v3)--(v4)--(v5)--(v1);
				
				\node[vertex] (s5) at (8,1.4) {};
				\node[vertex] (s41) at (6.4,1.2) {};
				\node[vertex] (s42) at (9.6,1.2) {};
				
				\draw (s5)--(s41) (s5)--(s42);;
				
				\foreach \i in {1,...,5}
				{
					\draw (s5)--(v\i);
				}
				\foreach \i in {1,3,4,5}
				{
					\draw (s41)--(v\i);
				}
				\foreach \i in {1,2,3,4}
				{
					\draw (s42)--(v\i);
				}
			\end{tikzpicture}
			\caption{$F4$.}
			\label{F4}
		\end{subfigure}%
		\begin{subfigure}[t]{.24\textwidth}
			\centering
			\begin{tikzpicture}[scale=0.8]
				\tikzstyle{vertex}=[draw, circle, fill=black!100, minimum width=1pt,inner sep=1pt]
				
				\node[vertex] (v1) at (8,1.7) {};
				\node[vertex] (v2) at (7,1) {};
				\node[vertex] (v3) at (7.5,0) {};
				\node[vertex] (v4) at (8.5,0) {};
				\node[vertex] (v5) at (9,1) {};
				
				\draw (v1)--(v2)--(v3)--(v4)--(v5)--(v1);
				
				\node[vertex] (s41) at (6.4,1.2) {};
				\node[vertex] (s42) at (9.6,1.2) {};
				\node[vertex] (s51) at (7.8,1.4) {};
				\node[vertex] (s52) at (8.2,1.4) {};
				
				\foreach \i in {1,...,5}
				{
					\draw (s51)--(v\i);
					\draw (s52)--(v\i);
				}
				\foreach \i in {1,3,4,5}
				{
					\draw (s41)--(v\i);
				}
				\foreach \i in {1,2,3,4}
				{
					\draw (s42)--(v\i);
				}
				\draw (s41)--(s51) (s42)--(s52);
				
			\end{tikzpicture}
			\caption{$F_5$.}
			\label{F5}
		\end{subfigure}%
		\begin{subfigure}[t]{.24\textwidth}
			\centering
			\begin{tikzpicture}[scale=0.8]
				\tikzstyle{vertex}=[draw, circle, fill=black!100, minimum width=1pt,inner sep=1pt]
				
				\node[vertex] (v1) at (8,1.7) {};
				\node[vertex] (v2) at (7,1) {};
				\node[vertex] (v3) at (7.5,0) {};
				\node[vertex] (v4) at (8.5,0) {};
				\node[vertex] (v5) at (9,1) {};
				
				\draw (v1)--(v2)--(v3)--(v4)--(v5)--(v1);
				
				\node[vertex] (s51) at (7.7,0.6) {};
				\node[vertex] (s52) at (8.3,0.6) {};
				
				\foreach \i in {1,...,5}
				{
					\draw (s51)--(v\i);
					\draw (s52)--(v\i);
				}
				
				\node[vertex] (s41) at (7.3,1.4) {};
				\node[vertex] (s42) at (8.7,1.4) {};
				
				\foreach \i in {1,2,5}
				{
					\draw (s41)--(v\i);
					\draw (s42)--(v\i);
				}
				\draw (s41)--(s51) (s42)--(s52) (s41)--(s42);
				
			\end{tikzpicture}
			\caption{$F_6$.}
			\label{F6}
		\end{subfigure}%
		\begin{subfigure}[t]{.24\textwidth}
			\centering
			\begin{tikzpicture}[scale=0.8]
				\tikzstyle{vertex}=[draw, circle, fill=black!100, minimum width=1pt,inner sep=1pt]
				
				\node[vertex] (v1) at (8,1.7) {};
				\node[vertex] (v2) at (7,1) {};
				\node[vertex] (v3) at (7.5,0) {};
				\node[vertex] (v4) at (8.5,0) {};
				\node[vertex] (v5) at (9,1) {};
				
				\draw (v1)--(v2)--(v3)--(v4)--(v5)--(v1);
				
				\node[vertex] (s51) at (7.7,0.6) {};
				\node[vertex] (s52) at (8.3,0.6) {};
				
				\foreach \i in {1,...,5}
				{
					\draw (s51)--(v\i);
					\draw (s52)--(v\i);
				}
				
				\node[vertex] (s41) at (7.3,1.4) {};
				\node[vertex] (s42) at (8.7,1.4) {};
				
				\foreach \i in {1,2,3}
				{
					\draw (s41)--(v\i);
				}
				\foreach \i in {1,4,5}
				{
					\draw (s42)--(v\i);
				}
				\draw (s41)--(s51) (s42)--(s52);
				
			\end{tikzpicture}
			\caption{$F_7$.}
			\label{F7}
		\end{subfigure}%
		\vspace{5mm}
		\begin{subfigure}[t]{.3\textwidth}
			\centering
			\begin{tikzpicture}[scale=0.8]
				\tikzstyle{vertex}=[draw, circle, fill=black!100, minimum width=1pt,inner sep=1pt]
				
				\node[vertex] (v1) at (8,2.2) {};
				\node[vertex] (v2) at (7,1.5) {};
				\node[vertex] (v3) at (7.5,0.5) {};
				\node[vertex] (v4) at (8.5,0.5) {};
				\node[vertex] (v5) at (9,1.5) {};
				
				\draw (v1)--(v2)--(v3)--(v4)--(v5)--(v1);
				
				\node[vertex] (s51) at (7.7,1) {};
				\node[vertex] (s52) at (8.3,1) {};
				
				\foreach \i in {1,...,5}
				{
					\draw (s51)--(v\i);
					\draw (s52)--(v\i);
				}
				
				\node[vertex] (s41) at (6.5,0) {};
				\node[vertex] (s42) at (9.5,0) {};
				
				\foreach \i in {2,3,4}
				{
					\draw (s41)--(v\i);
				}
				\foreach \i in {3,4,5}
				{
					\draw (s42)--(v\i);
				}
				\draw (s41)--(s51) (s42)--(s52) (s41)--(s42);
				
			\end{tikzpicture}
			\caption{$F_8$.}
			\label{F8}
		\end{subfigure}%
		\begin{subfigure}[t]{.3\textwidth}
			\centering
			\begin{tikzpicture}[scale=0.8]
				\tikzstyle{vertex}=[draw, circle, fill=black!100, minimum width=1pt,inner sep=1pt]
				
				\node[vertex] (v1) at (8,1.7) {};
				\node[vertex] (v2) at (7,0.8) {};
				\node[vertex] (v3) at (7.5,0) {};
				\node[vertex] (v4) at (8.5,0) {};
				\node[vertex] (v5) at (9,0.8) {};
				
				\draw (v1)--(v2)--(v3)--(v4)--(v5)--(v1);
				
				\node[vertex] (s321) at (6.3,1.2) {};
				\node[vertex] (s322) at (6.3,0.7) {};
				\node[vertex] (s351) at (9.7,1.2) {};
				\node[vertex] (s352) at (9.7,0.7) {};
				
				\foreach \i in {1,2,3}
				{
					\draw (s321)--(v\i);
					\draw (s322)--(v\i);
				}
				\foreach \i in {1,4,5}
				{
					\draw (s351)--(v\i);
					\draw (s352)--(v\i);
				}
				\draw (s321)--(s322) (s351)--(s352);
				
			\end{tikzpicture}
			\caption{$F_9$.}
			\label{F9}
		\end{subfigure}%
		
		\caption{Some 5-vertex-critical graphs.}
		\label{fig:5vertexcritical}
	\end{figure}
	
	\begin{theorem}\label{th}
		There are finitely many 5-vertex-critical $(P_5,bull)$-free graphs.
	\end{theorem}
	
	\begin{proof}
		Let $G=(V,E)$ be a 5-vertex-critical ($P_5$,bull)-free graph. We show that $|G|$ is bounded. If $G$ has a subgraph isomorphic to a member $F\in\mathcal{F}$, then $|V(G)|=|V(F)|$ by the definition of vertex-critical graph and so we are done. Hence, we assume in the following that $G$ has no subgraph isomorphic to a member in $\mathcal{F}$. Since there are exactly 13 5-vertex-critical $(P_5,C_5$)-free graphs \cite{HMRSV15}, the proof is completed if $G$ is $C_5$-free. So assume that $G$ contains an induced $C_5$ in the following. Let $C={v_1,v_2,v_3,v_4,v_5}$ be an induced $C_5$. We partition $V(G)$ with respect to $C$.
		
		\begin{claim}\label{S5} 
			$S_5$ is an independent set.
		\end{claim}
	
		\begin{proof}
			Suppose that $x,y\in S_5$ and $xy\in E$. Then $G$ contains $F_1$, a contradiction.
		\end{proof}

		\begin{claim}\label{coloring number}
			For each $1\le i\le 5$, some properties of $G$ are as follows:
			\begin{itemize}
				\item $\chi(G[S_3(i)])\le 2$.
				\item $\chi(G[S_2(i)\cup S_3(i)])\le 3$.
				\item $\chi(G[S_4(i)])\le 2$.
				\item $\chi(G[S_5\cup S_0])\le 4$.
			\end{itemize}
			
		\end{claim}
		
		\begin{proof}
			It suffices to prove for $i=1$. Suppose that $\chi(G[S_3(1)])\ge3$. Then $\chi(G-v_3)\ge 5$, contradicting that $G$ is 5-vertex-critical. So $\chi(G[S_3(1)])\le2$. Similarly, We can prove the other three properties.
		\end{proof}
		
		We first bound $S_0$.
		\begin{claim}\label{S0}
			{$N(S_0)\subseteq S_5$}.
		\end{claim}
		
		\begin{proof}
			Let $x\in N(S_0)$ and $y\in S_0$ be a neighbor of $x$. Then we show that $x\in S_5$. Let $1\le i\le5$. If $x\in S_2(i)\cup S_3(i)$, then $\{y,x,v_{i+1},v_{i+2},v_{i+3}\}$ induces a $P_5$. If $x\in S_4(i)$, then $\{v_i,v_{i+1},v_{i+2},x,y\}$ induces a bull. Therefore, $y\notin S_2\cup S_3\cup S_4$. It follows that $y\in S_5$.
		\end{proof}
		
		\begin{claim}\label{S0 color}
			If A is a component of $G[S_0]$, then $\chi(A)=4$.
		\end{claim}
		
		\begin{proof}
			By \autoref{coloring number}, $\chi(A)\le4$. Suppose that $\chi(A)\le 3$. So $\chi(C)\ge \chi(A)$. Combined with the fact that $C$ is anticomplete to $A$, we know that $C$ is not complete to $N(A)$ by \autoref{lem:XY}. This contradicts the facts that $C$ is complete to $S_5$ and $N(A)\subseteq S_5$. Thus $\chi(A)=4$.
		\end{proof}
		
		\begin{claim}\label{S0 connected}
			$G[S_0]$ is connected.
		\end{claim}
		
		\begin{proof}
			Suppose that there are two components $A_1$ and $A_2$ in $G[S_0]$. Since $G$ is connected, there must exist $w_1\in N(A_1)$ and so $w_1\in S_5$ by \autoref{S0}. By \autoref{coloring number}, $w_1$ cannot be complete to $A_1$ and $A_2$. So $w_1$ is mixed on an edge $x_1y_1\in E(A_1)$. Similarly, there exists $w_2\in S_5$ mixed on an edge $x_2y_2\in E(A_2)$ and not complete to $A_1$. So $w_2$ is anticomplete to $A_1$, otherwise if $w_2$ is mixed on an edge $z_1z_2\in E(A_1)$, then $\{z_1,z_2,w_2,x_2,y_2\}$ induces a $P_5$. It follows that $w_2$ is anticomplete to $\{x_1,y_1\}$. Then $\{y_1,x_1,w_1,v_1,w_2\}$ induces a $P_5$, a contradiction.
		\end{proof}
		
		By Claims \ref{S0 color}-\ref{S0 connected}, we obtain the following claim. 
		
		\begin{claim}\label{S0 4-chromatic}
			$G[S_0]$ is a connected 4-chromatic graph.
		\end{claim}
		
		\begin{claim}
			$N(S_0)=S_5$.
		\end{claim}
		
		\begin{proof}
			Suppose that $w_1\in S_5$ is anticomplete to $S_0$. Since $G$ is connected, there must exist $w_2\in S_5$, which is a neighbor of $S_0$. By \autoref{coloring number}, $w_2$ is not complete to $S_0$ and so mixed on an edge $xy$ in $G[S_0]$. Thus, $\{w_1,v_1,w_2,x,y\}$ induces a $P_5$, a contradiction.
		\end{proof}
		
		To bound $S_0$, we partition $S_0$ into two parts. Let $L=S_0\cap N(S_5)$ and $R=S_0\setminus L$. 
		
		\begin{claim}\label{L S5}
			If $R\neq\emptyset$, then (i)$L$ is complete to $S_5$; (ii)$N(R)=L$.
		\end{claim}
		
		\begin{proof}
			Let $L_i=\{l\in L|d(l,R)=i\}$, where $i\ge 1$. Let $l\in L_1$. There exists $r\in R$, which is adjacent to $l$. Let $u\in S_5$ be a neighbor of $l$. Note that if $|S_5|=1$, $S_5$ is a clique cutset of $G$, contradicting \autoref{lem:xy}. So $|S_5|\ge 2$. For each $u'\in S_5\setminus\{u\}$, $u'$ is adjacent to $l$, otherwise $\{r,l,u,v_1,u'\}$ induces a $P_5$. Hence, $L_1$ is complete to $S_5$. Let $l_2\in L_2$. By the definition of $L_2$, there must exist $l_1\in L_1$, $l_2$ is adjacent to $l_1$. Let $r_1\in R$ and $u_2\in S_5$ be the neighbor of $l_1$ and $l_2$, respectively. Since $d(l_2,R)=2$, $l_2r_1\notin E$. Since $L_1$ is complete to $S_5$, $l_1u_2\in E$. Thus $\{v_1,u_2,l_2,l_1,r_1\}$ induces a bull, a contradiction. So $L_2=\emptyset$ and thus $L_i=\emptyset$ for each $i\ge 3$. Then $L=L_1$. Therefore, $L$ is complete to $S_5$ and $N(R)=L$.
		\end{proof}
		
		\begin{claim}\label{LR components}
			Let $L'$ and $R'$ be components of $G[L]$ and $G[R]$, respectively. Then $L'$ is complete or anticomplete to $R'$.
		\end{claim}
		
		\begin{proof}
			Let $u\in S_5$. By \autoref{L S5}, $u$ is complete to $L'$. Assume $L'$ is not anticomplete to $R'$. We show that $L'$ is complete to $R'$ in the following. Let $l_1\in V(L')$ and $r_1\in V(R')$ be adjacent. If $l_1$ is mixed on $R'$, then $l_1$ must be mixed on an edge $x_1y_1$ in $R'$ and so $\{v_1,u,l_1,x_1,y_1\}$ induces a $P_5$, a contradiction. So $l_1$ is complete to $R'$. Suppose that $l_2\in V(L')$ is not complete to $R'$, then there exists $r_2\in V(R')$ not adjacent to $l_2$. Since $l_1r_2\in E$, $r_2$ is mixed on $L'$ and so mixed on an edge $x_2y_2$ in $L'$. Thus $\{v_1,u,x_2,y_2,r_2\}$ induces a bull, a contradiction. It follows that $L'$ is complete to $R'$.
		\end{proof}
		
		\begin{claim}\label{R finite}
			$|R|\le 8$.
		\end{claim}
		
		\begin{proof}
			Let $R'$ and $R''$ be two arbitrary components of $G[R]$. Let $u_1\in S_5$. If there exists $l_1,l_2\in L$ such that $l_1\in N(R')\setminus N(R'')$ and $l_2\in N(R'')\setminus N(R')$, then $\{u_1,l_1,l_2\}\cup R'\cup R''$ contains an induced bull or an induced $P_5$, depending on whether $l_1l_2\in E$. So $N(R')\subseteq N(R'')$ or $N(R'')\subseteq N(R')$. We may assume $N(R')\subseteq N(R'')$. By \autoref{LR components}, $R''$ is complete to $N(R')$. It follows from \autoref{lem:XY} that $\chi(R'')<\chi(R')$. By \autoref{S0 4-chromatic} and \autoref{LR components}, for each component of $G[R]$, there must exist a vertex in $L$ complete to this component. Since $G[S_0]$ is 4-chromatic, the chromatic number of components of $G[R]$ is at most 3. So there are at most three components $R_1,R_2$ and $R_3$ in $G[R]$. Assume that $\chi(R_1)=1,\chi(R_2)=2$ and $\chi(R_3)=3$. By \autoref{LR components} and the definition of $R$, we know that $R_1,R_2$ and $R_3$ are all homogeneous. By \autoref{lem:homogeneous}, we know that $|R_1|=1$, $|R_2|=2$ and $|R_3|\le 5$. Therefore, $|R|\le 8$.
		\end{proof}
		
		\begin{claim}\label{L finite}
			If $R\neq \emptyset$, then $|L|\le 8$.
		\end{claim}
		
		\begin{proof}
			Let $L'$ and $L''$ be two arbitrary components of $G[L]$. By \autoref{L S5}, $L',L''\subseteq N(R)$. Let $u_1\in S_5$. By \autoref{L S5}, \autoref{LR components} and \autoref{coloring number}, each component of $G[L]$ must be complete to some component of $G[R]$ and so $\chi(G[L])\le 3$. Suppose that there exists $r_1,r_2\in R$ such that $r_1\in N(L')\setminus N(L'')$ and $r_2\in N(L'')\setminus N(L')$. Then $r_1$ and $r_2$ belong to different components of $R$ by \autoref{LR components}. So $r_1r_2\notin E$. Then $\{u_1,r_1,r_2\}\cup L'\cup L''$ contains an induced $P_5$, a contradiction. Combined with \autoref{L S5}, we know that $N(L')\subseteq N(L'')$ or $N(L'')\subseteq N(L')$. We may assume $N(L')\subseteq N(L'')$. By \autoref{LR components}, $L''$ is complete to $N(L')$. It follows from \autoref{lem:XY} that $\chi(L'')<\chi(L')$. Note that $\chi(G[L])\le 3$. So there are at most three components $L_1,L_2$ and $L_3$ in $G[L]$. Assume that $\chi(L_1)=1,\chi(L_2)=2$ and $\chi(L_3)=3$. By \autoref{LR components} and \autoref{L S5}, we know that $L_1,L_2$ and $L_3$ are all homogeneous. By \autoref{lem:homogeneous}, we know that $|L_1|=1$, $|L_2|=2$ and $|L_3|\le 5$. Therefore, $|L|\le 8$.
		\end{proof}
		
		By Claims \ref{R finite}-\ref{L finite}, we obtain the following claim.
		
		\begin{claim}\label{cla:S0 1}
			If $R\neq\emptyset$, $|S_0|\le 16$.
		\end{claim}
		
		Next, we bound $S_0$ when $R=\emptyset$.
		
		\begin{claim}\label{cla:S0 2}
			If $R=\emptyset$, then $|S_0|\le 13$.
		\end{claim}
		
		\begin{proof}
			Since $R=\emptyset$, $S_0\subseteq N(S_5)$. For each $v\in S_0$, $\chi(G-v)=4$ since $G$ is 5-vertex-critical. Let $\pi$ be a 4-coloring of $G-v$. By the fact that $\chi(C)=3$ and $S_5$ is complete to $C$, all vertices in $S_5$ must be colored with the same color in $\pi$. Since $S_0\subseteq N(S_5)$, the vertices in $S_0\setminus\{v\}$ must be colored with the remaining three colors, i.e., $\chi(G[S_0]-v)\le 3$. Combined with \autoref{S0 4-chromatic}, $G[S_0]$ is a $P_5$-free 4-vertex-critical graph. By \autoref{thm:finite4Critical}, $|S_0|\le 13$.
		\end{proof}
		
		By Claims \ref{cla:S0 1}-\ref{cla:S0 2}, $|S_0|\le 16$. Next, we bound $S_5$.
		
		\begin{claim}\label{S4(i)S5}
			For at most one value of $i$, where $1\le i\le 5$, $S_4(i)$ is not anticomplete to $S_5$.
		\end{claim}
		
		\begin{proof}
			Suppose that $S_4(i)$ and $S_4(j)$ are not anticomplete to $S_5$, where $1\le i<j\le5$. Then $G$ must have a subgraph isomorphic to $F_2,F_3,F_4$ or $F_5$, a contradiction.
		\end{proof}
		
		\begin{claim}
			$|S_5|\le 2^{16}$.
		\end{claim}
		
		\begin{proof}
			Suppose that $|S_5|> 2^{|S_0|}$. By the pigeonhole principle, there are two vertices $u,v\in S_5$ that have the same neighborhood in $S_0$. Since $u$ and $v$ are not comparable, there exists $x\in N(u)\setminus N(v)$ and $y\in N(v)\setminus N(u)$. Clearly, $x,y\in S_3\cup S_4(i)$ by \autoref{S4(i)S5} and \ref{S2S5}, for some $1\le i\le 5$. By symmetry, we assume $i=1$. 
			
			Suppose that $x,y\in S_4(1)$. Then $xy\notin E$, otherwise $G$ has a subgraph isomorphic to $F_8$. So $\{x,u,v_1,v,y\}$ induces a $P_5$, a contradiction.
			
			Suppose that $x,y\in S_3$. Without loss of generality, we assume $x\in S_3(1)$. If $y\in S_3(3)\cup S_3(4)$, $G$ must have a subgraph isomorphic to $F_7$, a contradiction. If $y\in S_3(2)\cup S_3(5)$, then $xy\in E$ by \ref{S3(i)S3(i+1)} and so $G$ contains $F_8$, a contradiction. If $y\in S_3(1)$, then $xy\notin E$, otherwise $G$ has a subgraph isomorphic to $F_6$. Then $\{x,u,v_3,v,y\}$ induces a $P_5$, a contradiction. 
			
			So we assume that $x\in S_4(1)$ and $y\in S_3$. If $y\in S_3(1)\cup S_3(2)\cup S_3(5)$, then $G$ has a subgraph isomorphic to $F_7$, a contradiction. Thus $y\in S_3(3)\cup S_3(4)$. From \ref{S2(i)S4(i+2)} we know that $xy\in E$. Note that $G$ has a subgraph isomorphic to $F_8$, a contradiction.
			
			Therefore, $|S_5|\le 2^{|S_0|}\le 2^{16}$.
		\end{proof}
		
		Next, we bound $S_2$. By \ref{S2(i)S2(i+1)}-\ref{S2S5} and \autoref{S0}, for each $1\le i\le 5$, all vertices in $V\setminus S_2(i)$  are complete or anticomplete to $S_2(i)$, except those in $S_3(i)$. So we divide $S_2(i)$ into two parts. Let $R(i)=S_2(i)\cap N(S_3(i))$ and $L(i)=S_2(i)\setminus R(i)$.
		
		\begin{claim}\label{P3 conclusion1}
			If $G[R(i)]$ contains a $P_3$, then the two endpoints of the $P_3$ have the same neighborhood in $S_3(i)$.
		\end{claim}
		
		\begin{proof}
			Let $uvw$ be a $P_3$ contained in $R(i)$. Let $u'\in S_3(i)$ be a neighbor of $w$. Then $uu'\in E$, otherwise $\{u,v,w,u',v_i\}$ induces a bull or a $P_5$, depending on whether $vu'\in E$. So $N_{S_3(i)}(w)\subseteq N_{S_3(i)}(u)$. Similarly,  $N_{S_3(i)}(u)\subseteq N_{S_3(i)}(w)$. Therefore, $u$ and $w$ have the same neighborhood in $S_3(i)$. 
		\end{proof}
		
		\begin{claim}\label{L(i)}
			$|L(i)|\le 8$.
		\end{claim}
		
		\begin{proof}
			If $S_3(i)=\emptyset$ or $R(i)=\emptyset$, then $S_2(i)$ is homogeneous. If there are two components $X$ and $Y$ in $G[S_2(i)]$, then $Y$ is complete to $N(X)$ and $X$ is complete to $N(Y)$, contradicting \autoref{lem:XY}. So $G[S_2(i)]$ is connected. By \autoref{coloring number} and \autoref{lem:homogeneous}, $G[S(i)]$ is a $K_1$, a $K_2$, a $K_3$ or a $C_5$. Thus $|L(i)|\le 5$. 
			
			So we assume that $S_3(i)\neq \emptyset$ and $R(i)\neq \emptyset$. Let $u$ be an arbitrary vertex in $R(i)$ and $u'$ be its neighbor in $S_3(i)$. Then $u$ is not mixed on any edge $xy$ in $L(i)$, otherwise $\{y,x,u,u',v_i\}$ induces a $P_5$. Then $u$ is complete or anticomplete to any component of $L(i)$ and so all components of $L(i)$ are homogeneous. By \autoref{lem:homogeneous}, each component of $L(i)$ is a $K_1$, a $K_2$, a $K_3$ or a $C_5$. 
			
			We show that there is at most one 3-chromatic component in $L(i)$. Suppose that $X_1$ and $Y_1$ are two 3-chromatic components in $L(i)$. Note that $X_1$ and $Y_1$ are homogeneous. Since $\chi(G[S_2(i)])\le 3$, $X_1$ and $Y_1$ are anticomplete to $R(i)$. So $Y_1$ is complete to $N(X_1)$ and $X_1$ is complete to $N(Y_1)$, which contradicts \autoref{lem:XY}. So, there is at most one 3-chromatic component in $L(i)$.
			
			Then we show that there is at most one $K_2$-component in $L(i)$. Suppose that $X_2=x_1y_1$ and $Y_2=x_2y_2$ are two $K_2$-components in $L(i)$. Note that $X_2$ and $Y_2$ are homogeneous. By \autoref{lem:XY}, there must exist $u_1,u_2\in R(i)$ such that $u_1$ is complete to $X_2$ and anticomplete to $Y_2$ and $u_2$ is complete to $Y_2$ and anticomplete to $X_2$ . Let $u_1',u_2'\in S_3(i)$ be the neighbor of $u_1$ and $u_2$, respectively. Clearly, $u_1'$ and $u_2'$ are not the same vertex, otherwise $\{x_1,u_1,u_1',u_2,x_2\}$ induces a bull or a $P_5$, depending on whether $u_1u_2\in E$. So $u_1'u_2\notin E$ and $u_2'u_1\notin E$. It follows that $u_1u_2\notin E$, otherwise $\{x_2,u_2,u_1,u_1',v_i\}$ induces a $P_5$. Then $\{u_1,u_1',v_i,u_2',u_2\}$ induces a bull or a $P_5$, depending on whether $u_1'u_2'\in E$, a contradiction. So, there is at most one $K_2$-component in $L(i)$.
			
			Similarly, there is at most one $K_1$-component in $L(i)$. It follows that $|L(i)|\le 8$. The proof is completed.
		\end{proof}
		
		\begin{figure}[h]
			\centering
			\begin{tikzpicture}[scale=0.8]
				\tikzstyle{vertex}=[draw, circle, fill=black!100, minimum width=1pt,inner sep=1pt]
				
				\node[vertex,label=below:$v$] (v) at (1.5,0) {};
				\node[vertex,label=left:$u$] (u) at (0,1) {};
				\node[vertex,label=right:$w$] (w) at (3,1) {};
				\node[vertex,label=left:$s$] (s) at (0,2.5){};
				\node[vertex,label=right:$t$] (t) at (3,2.5) {};
				\draw (s)--(u)--(v)--(s) (v)--(w)--(t)--(v);
				
			\end{tikzpicture}
			
			\caption{The graph contained in $G[R(i)]$.}
			\label{fig:uvwst}
		\end{figure}
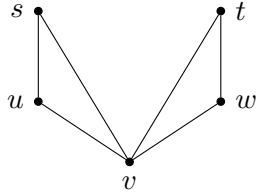
		
		\begin{claim}\label{P3 conclusion2}
			If $G[R(i)]$ contains $P_3=uvw$, then $G[R(i)]$ must contain the graph induced by $\{u,v,w,s,t\}$ in \autoref{fig:uvwst}. Moreover, $u,w,s$ and $t$ have the same neighborhood in $S_3(i)$ and $N_{S_3(i)}(u)\cap N_{S_3(i)}(v)=\emptyset$.
		\end{claim}
		
		\begin{proof}
			Let $u'$ be an arbitrary neighbor of $w$ in $S_3(i)$. By \autoref{P3 conclusion1} we know that $N_{S_3(i)}(u)=N_{S_3(i)}(w)$ and so $uu'\in E$. Since $u$ and $w$ are not comparable, there must exist $s\in N(u)\setminus N(w)$ and $t\in N(w)\setminus N(u)$. Clearly, $s,t\in L(i)\cup R(i)$.
			
			\noindent{\bf Case 1. }$s,t\in L(i)$. Then $st\notin E$, otherwise $\{s,t,w,u',v_i\}$ induces a $P_5$. Moreover, $sv\notin E$, otherwise $\{s,v,w,u',v_i\}$ induces a bull or a $P_5$, depending on whether $vu'\in E$. Similarly, $tv\notin E$. So $\{s,u,v,w,t\}$ induces a $P_5$, a contradiction.
			
			\noindent{\bf Case 2. }One vertex of $\{s,t\}$ belongs to $L(i)$ and the other belongs to $R(i)$. We assume that $s\in L(i)$ and $t\in R(i)$. Then $sv\notin E$, otherwise $\{s,v,w,u',v_i\}$ induces a bull or a $P_5$, depending on whether $vu'\in E$. So $vu'\notin E$, otherwise $\{s,u,v,u',v_i\}$ induces a bull. Let $z'$ be a neighbor of $v$ in $S_3(i)$. Clearly, $\{s,u,v,z',v_i\}$ induces a bull or a $P_5$, depending on whether $uz'\in E$, a contradiction.
			
			\noindent{\bf Case 3. }$s,t\in R(i)$. Suppose that $sv\notin E$. Then $suv$ is a $P_3$ and so $u'$ is complete or anticomplete to $\{s,v\}$ by \autoref{P3 conclusion1}. Suppose that $u'$ is complete to $\{s,v\}$. If $vt\in E$, then $uvt$ is a $P_3$ and so $tu'\in E$ by \autoref{P3 conclusion1}. Then $\{t,v,w,u'\}$ induces a $K_4$, contradicting that $\chi(G[S_2(i)\cup S_3(i)])\le 3$. So $vt\notin E$. Hence $vwt$ is a $P_3$ and then $tu'\in E$ by \autoref{P3 conclusion1}. Then $st\in E$, otherwise $\{s,u,v,w,t\}$ induces a $P_5$. It is easy to verify that $\{s,u,v,w,t,u'\}$ induces a 4-chromatic subgraph, contradicting that $\chi(G[S_2(i)\cup S_3(i)])\le 3$. So $u'$ must be anticomplete to $\{s,v\}$. Then $st\notin E$, otherwise $\{s,t,w,u',v_i\}$ induces a bull or a $P_5$, depending on whether $tu'\in E$. Hence $tv\in E$, otherwise $\{s,u,v,w,t\}$ induces a $P_5$. Let $z'$ be an arbitrary neighbor of $v$ in $S_3(i)$. Since $suv$ is a $P_3$, $sz'\in E$ by \autoref{P3 conclusion1}. Note that $uvt$ and $uvw$ are all $P_3$ and so $N_{S_3(i)}(u)=N_{S_3(i)}(w)=N_{S_3(i)}(t)$. Then $tz'\notin E$, otherwise $\{t,v,z',w\}$ induces a $K_4$. Note that $\{s,z',v_i,u',w\}$ induces a bull or a $P_5$, depending on whether $u'z'\in E$, a contradiction. Thus $sv\in E$. By symmetry, $tv\in E$. 
			
			Since $svw$ and $uvt$ are all $P_3$, we know that $u,w,s,t$ have the same neighborhood in $S_3(i)$ by \autoref{P3 conclusion1} and so $su',tu'\in E$. Then $vu'\notin E$, otherwise $\{v,w,t,u'\}$ induces a $K_4$. Since $u'$ is an arbitrary neighbor of $w$ in $S_3(i)$, $v$ is anticomplete to $N_{S_3(i)}(u)$. Thus $N_{S_3(i)}(u)\cap N_{S_3(i)}(v)=\emptyset$. 
			
			If $st\in E$, then $ust$ is a $P_3$. From the above proof we know that $s$ is anticomplete to $N_{S_3(i)}(u)$, which contradicts the fact that $su'\in E$. So $st\notin E$. It follows that $\{u,v,w,s,t\}$ induces the graph in \autoref{fig:uvwst}. This completes the proof of the claim.
		\end{proof}
		
		\begin{claim}\label{P3-free}
			$G[R(i)]$ is $P_3$-free.
		\end{claim}
		
		\begin{proof}
			Suppose that $G[R(i)]$ contains a $P_3=uvw$. By \autoref{P3 conclusion2}, $G[R(i)]$ contains a subgraph in \autoref{fig:uvwst} induced by $\{u,v,w,s,t\}$. Moreover, $u,w,s,t$ have the same neighborhood in $S_3(i)$ and $v$ is anticomplete to $N_{S_3(i)}(u)$. Let $u'$ and $v'$ be arbitrary neighbor of $u$ and $v$ in $S_3(i)$, respectively. Then $u'$ is complete to $\{u,w,s,t\}$ and nonadjacent to $v$ and $v'$ is anticomplete to $\{u,w,s,t\}$. It follows from \autoref{lem:XY} that $\{w,t\}$ is not complete to $N\{u,s\}$. So there exists $a\in N\{u,s\}$ such that $a$ is not complete to $\{w,t\}$. Clearly, $a\in L(i)\cup R(i)$. 
			
			Suppose $a\in L(i)$. Assume that $as\in E$. So $au\in E$, otherwise $\{a,s,u,u',v_i\}$ induces a bull. Then $av\in E$, otherwise $\{a,u,v,v',v_i\}$ induces a $P_5$. Note that $\{a,s,v,u\}$ induces a $K_4$, a contradiction. Thus $a\in R(i)$.
			
			If $a$ is adjacent to only one vertex in $\{s,u\}$, then either $usa$ or $sua$ is a $P_3$ and so $N_{S_3(i)}(s)\cap N_{S_3(i)}(u)=\emptyset$ by \autoref{P3 conclusion2}, contradicting that $su',uu'\in E$. Thus $a$ is complete to $\{s,u\}$. Then $av\notin E$, otherwise $\{s,u,a,v\}$ induces a $K_4$. Because $auv$ is a $P_3$, we know that $au'\notin E$ and $av'\in E$ by \autoref{P3 conclusion2}. Since $a$ is not complete to $\{w,t\}$, we assume that $at\notin E$ by symmetry. Note that $\{t,u',v_i,v',a\}$ induces a bull or a $P_5$, depending on whether $u'v'\in E$, a contradiction.
			
			Therefore, $G[R(i)]$ is $P_3$-free.
		\end{proof}
		
		Since $G[R(i)]$ is $P_3$-free, $G[R(i)]$ is a disjoint union of cliques. By \autoref{coloring number}, each component of $G[R(i)]$ is a $K_1$, a $K_2$ or a $K_3$. We next prove that the number of them is finite.
		
		\begin{claim}\label{R(i) 1}
			There are at most $2^{|L(i)|}$ $K_1$-components and 5 $K_2$-components in $G[R(i)]$.
		\end{claim}
		
		\begin{proof}
			We first show that there are at most $2^{|L(i)|}$ $K_1$-components in $G[R(i)]$. Suppose there are more than $2^{|L(i)|}$ $K_1$-components in $G[R(i)]$. By the pigeonhole principle, there exists $u,v\in R(i)$ and they have the same neighborhood in $L(i)$. Since $u$ and $v$ are not comparable, there exists $u',v'\in S_3(i)$ such that $u'\in N(u)\setminus N(v)$ and $v'\in N(v)\setminus N(u)$. Then $\{u,u',v_i,v',v\}$ induces a bull or a $P_5$, depending on whether $u'v'\in E$, a contradiction. So there are at most $2^{|L(i)|}$ $K_1$-components in $G[R(i)]$.
			
			Next we show that there are at most 5 $K_2$-components in $G[R(i)]$.
			
			Suppose that $A_1$ and $A_2$ are two homogeneous $K_2$-components of $G[R(i)]$. By \autoref{lem:XY}, there exists $x_1\in N(A_1)\setminus N(A_2)$ and $y_1\in N(A_2)\setminus N(A_1)$. Clearly, $x_1,y_1\in S_3(i)\cup L(i)$. Suppose that $x_1,y_1\in L(i)$. Let $w_1,w_2\in S_3(i)$ be the neighbor of $A_1$ and $A_2$, respectively. If $x_1y_1\in E$, then $\{y_1,x_1,w_1,v_i\}\cup A_1$ contains an induced $P_5$. So $x_1y_1\notin E$. Note that $w_2\notin N(A_1)$, otherwise $\{w_2,x_1,y_1\}\cup A_1\cup A_2$ contains an induced $P_5$. Similarly, $w_1\notin N(A_2)$. Then $\{v_i,w_1,w_2\}\cup A_1\cup A_2$ contains an induced bull or an induced $P_5$, depending on whether $w_1w_2\in E$, a contradiction. Suppose that $x_1\in L(i)$ and $y_1\in S_3(i)$. Let $w_3$ be the neighbor of $A_1$ in $S_3(i)$. Note that $w_3\in N(A_2)$, otherwise $\{v_i,w_3,y_1\}\cup A_1\cup A_2$ contains an induced bull or an induced $P_5$, depending on whether $w_3y_1\in E$. Then $w_3y_1\in E$, otherwise $\{x_1,y_1,w_3\}\cup A_1\cup A_2$ contains an induced $P_5$. Then $\{w_3,y_1\}\cup A_2$ induces a $K_4$, contradicting that $\chi(G[S_2(i)\cup S_3(i)])\le 3$. So $x_1,y_1\in S_3(i)$ and then $\{v_i,x_1,y_1\}\cup A_1\cup A_2$ contains an induced bull or an induced $P_5$, depending on whether $x_1y_1\in E$, a contradiction. Thus there is at most one homogeneous $K_2$-component in $G[R(i)]$.
			
			Let $B_1=x_3y_3$ and $B_2=x_4y_4$ be two arbitrary nonhomogeneous $K_2$-components of $G[R(i)]$ and the vertices mixed on $B_1$ or $B_2$ are clearly in $L(i)\cup S_3(i)$. Suppose that each vertex in $S_3(i)$ is complete or anticomplete to $B_1$, then there exists $z'\in L(i)$ mixed on $B_1$. Let $t\in S_3(i)$ be complete to $B_1$, then $\{z',x_3,y_3,t,v_i\}$ induces a bull, a contradiction. So there must exist $z_3\in S_3(i)$ mixed on $B_1$. Similarly, there exists $z_4\in S_3(i)$ mixed on $B_2$. By symmetry, we assume $z_3x_3,z_4x_4\in E$ and $z_3y_3,z_4y_4\notin E$. Then $z_3$ is complete or anticomplete to $B_2$, otherwise $\{y_3,x_3,z_3,x_4,y_4\}$ induces a $P_5$. Similarly, $z_4$ is complete or anticomplete to $B_1$. If $z_3$ is anticomplete to $B_2$ and $z_4$ is anticomplete to $B_1$, then $\{x_3,z_3,v_i,z_4,x_4\}$ induces a bull or a $P_5$, depending on whether $z_3z_4\in E$. If $z_3$ is complete to $B_2$ and $z_4$ is complete to $B_1$, then $\{y_3,z_4,v_i,z_3,y_4\}$ induces a bull or a $P_5$, depending on whether $z_3z_4\in E$. So we assume $z_3$ is anticomplete to $B_2$ and $z_4$ is complete to $B_1$. It follows that $z_3z_4\in E$, otherwise $\{y_4,x_4,z_4,v_i,z_3\}$ induces a $P_5$. So there are at most 4 nonhomogeneous $K_2$-components in $R(i)$, otherwise the vertices in $S_3(i)$ mixed on them respectively can induce a $K_5$, a contradiction.
			
			The above proof shows that there are at most $2^{|L(i)|}$ $K_1$-components and 5 $K_2$-components in $G[R(i)]$.
		\end{proof}
		
		\begin{claim}\label{R(i) 2}
			There is at most one $K_3$-component in $G[R(i)]$.
		\end{claim}
		
		\begin{proof}
			Suppose that $T_1=x_1y_1z_1,T_2=x_2y_2z_2$ are two arbitrary $K_3$-components of $G[R(i)]$. Let $x',y'\in S_3(i)$ be the neighbor of $T_1$ and $T_2$, respectively. Since $\chi(G[S_2(i)\cup S_3(i)])\le 3$, $x'$ is mixed on $T_1$ and $y'$ is mixed on $T_2$. By symmetry, we assume that $x'x_1,y'x_2\in E$ and $x'y_1,y'y_2\notin E$. So $x'$ is not mixed on $T_2$, otherwise $\{y_1,x_1,x'\}\cup T_2$ contains an induced $P_5$. Moreover, since $\chi(G[S_2(i)\cup S_3(i)])\le 3$, $x'$ is not complete to $T_2$. Thus $x'$ is anticomplete to $T_2$. Similarly, $y'$ is anticomplete to $T_1$. Then $\{x_1,x',v_i,y',x_2\}$ induces a bull or a $P_5$, depending on whether $x'y'\in E$, a contradiction.
			
			Therefore, there is at most one $K_3$-component in $G[R(i)]$.
		\end{proof}
		
		By Claims \ref{L(i)}, \ref{R(i) 1} and \ref{R(i) 2}, $|L(i)|\le 8$ and $|R(i)|\le 2^{|L(i)|}+13$. So $|S_2|\le 5\times(2^8+21)$.
		
		Finally, we bound $S_3$ and $S_4$.
		
		\begin{claim}\label{S3 trivial}
			For each $1\le i\le 5$, the number of $K_1$-components in $G[S_3(i)]$ is not more than $2^{|S_2(i)\cup S_5|}$.
		\end{claim}
		
		\begin{proof}
			It suffices to prove for $i=1$. Suppose that the number of $K_1$-components in $G[S_3(1)]$ is more than $2^{|S_2(1)\cup S_5|}$. The pigeonhole principle shows that there are two $K_1$-components $u,v$ having the same neighborhood in $S_2(1)\cup S_5$. Since $u$ and $v$ are not comparable, there must exist $u'\in N(u)\setminus N(v)$ and $v'\in N(v)\setminus N(u)$. By \ref{S2(i)S2(i+1)}, \ref{S2(i)S3(i+2)}, \ref{S3(i)S3(i+1)} and \ref{S2(i)S4(i+2)}, $u',v'\in S_3(3)\cup S_3(4)\cup S_4(1)\cup S_4(2)\cup S_4(5)$. So $\{u,u',v_3,v',v\}$ induces a bull or a $P_5$, depending on whether $u'v'\in E$, a contradiction.
		\end{proof}
		
		\begin{claim}\label{S4 trivial}
			For each $1\le i\le 5$, the number of $K_1$-components in $G[S_4(i)]$ is not more than $2^{|S_5|}$.
		\end{claim}
		
		\begin{proof}
			It suffices to prove for $i=1$. Suppose that the number of $K_1$-components in $G[S_4(1)]$ is more than $2^{|S_5|}$. The pigeonhole principle shows that there are two $K_1$-components $u,v$ having the same neighborhood in $S_5$. Since $u$ and $v$ are not comparable, there must exist $u'\in N(u)\setminus N(v)$ and $v'\in N(v)\setminus N(u)$. By \ref{S2(i)S4(i)}, \ref{S2(i)S4(i+2)} and \ref{S2S5}, $u',v'\in (\cup_{i=1,2,5}S_3(i))\cup (\cup_{2\le i\le 5}S_4(i))$. So $\{u,u',v_1,v',v\}$ induces a bull or a $P_5$, depending on whether $u'v'\in E$, a contradiction.
		\end{proof}
		
		\begin{claim}\label{S4(i) 2-chromatic}
			If $\chi(S_4(i))=2$ for some $1\le i\le 5$, then $S_3\cup S_4$ is bounded.
		\end{claim}
		
		\begin{proof}
			Without loss of generality, we assume $\chi(S_4(1))=2$. It follows from \ref{S2(i)S4(i+2)} that $S_3(3)=S_3(4)=\emptyset$, otherwise $S_4(1)\cup S_3(3)\cup \{v_3,v_4\}$ contains an induced $K_5$. Since $G$ has no subgraph isomorphic to $F_9$, $\chi(S_4(i))\le 1$ for each $2\le i\le 5$ and $\chi(S_3(j))\le 1$ for each $j=1,2,5$. By Claims \ref{S3 trivial}-\ref{S4 trivial}, $S_3\cup(\cup_{2\le i\le 5}S_4(i))$ is bounded and the number of $K_1$-components in $G[S_4(1)]$ is also bounded.
			
			We now show that the number of vertices in a 2-chromatic component of $G[S_4(1)]$ is bounded. Let $A$ be a 2-chromatic component of $G[S_4(1)]$ and so $A$ is bipartite. Let the bipartition of $A$ be $(X,Y)$. Suppose that $|X|>2^{|S_3\cup (\cup_{2\le i\le 5}S_4(i))\cup S_5|}$. By the pigeonhole principle, there exists two vertices $x_1,x_2\in X$ which have the same neighborhood in $S_3\cup (\cup_{2\le i\le 5}S_4(i))\cup S_5$. Since $x_1$ and $x_2$ are not comparable, there must exist $y_1\in N(x_1)\setminus N(x_2),y_2\in N(x_2)\setminus N(x_1)$. Clearly, $y_1,y_2\in Y$ and so $\{x_1,x_2,y_1,y_2\}$ induces a $2K_2$ in $A$. Since $A$ is connected and bipartite, $A$ contains a $P_5$ by \autoref{2K2}, a contradiction. Thus $|X|\le 2^{|S_3\cup (\cup_{2\le i\le 5}S_4(i))\cup S_5|}$. Similarly, $|Y|\le 2^{|S_3\cup (\cup_{2\le i\le 5}S_4(i))\cup S_5|}$. Thus the number of vertices in $A$ is bounded.
			
			Then we show that there are at most five 2-chromatic components in $G[S_4(1)]$. 
			
			Suppose that $A_1$ and $A_2$ are two homogeneous 2-chromatic components of $G[S_4(1)]$. By \autoref{lem:XY}, $A_1$ is not complete to $N(A_2)$ and $A_2$ is not complete to $N(A_1)$. So there must exist $z_1\in N(A_1)\setminus N(A_2)$ and $z_2\in N(A_2)\setminus N(A_1)$. Clearly, $z_1,z_2\in (\cup_{i=1,2,5}S_3(i))\cup (\cup_{2\le i\le 5}S_4(i))\cup S_5$. Then $\{v_1,z_1,z_2\}\cup A_1\cup A_2$ contains an induced bull or an induced $P_5$, depending on whether $z_1z_2\in E$, a contradiction. Thus there is at most one homogeneous 2-chromatic component in $G[S_4(1)]$.
			
			Let $B_1,B_2$ be two nonhomogeneous 2-chromatic components of $G[S_4(1)]$. So there exists $x'$ mixed on $B_1$ and $y'$ mixed on $B_2$. Let $x'$ be mixed on edge $x_3y_3$ in $B_1$ and $y'$ be mixed on edge $x_4y_4$ in $B_2$. By symmetry, assume that $x'x_3,y'x_4\in E$ and $x'y_3,y'y_4\notin E$. It is evident that $x'$ and $y'$ are not the same vertex, otherwise $\{y_3,x_3,x',x_4,y_4\}$ induces a $P_5$. Similarly, $x'$ is not mixed on $x_4y_4$ and $y'$ is not mixed on $x_3y_3$. Clearly, $x',y'\in (\cup_{i=1,2,5}S_3(i))\cup (\cup_{2\le i\le 5}S_4(i))\cup S_5$. If $x'$ is anticomplete to $\{x_4,y_4\}$ and $y'$ is anticomplete to $\{x_3,y_3\}$, then $\{x_3,x',v_1,y',x_4\}$ induces a bull or a $P_5$, depending on whether $x'y'\in E$. If $x'$ is complete to $\{x_4,y_4\}$ and $y'$ is complete to $\{x_3,y_3\}$, then $\{y_4,x',v_1,y',y_3\}$ induces a bull or a $P_5$, depending on whether $x'y'\in E$. So we assume that $x'$ is complete to $\{x_4,y_4\}$ and $y'$ is anticomplete to $\{x_3,y_3\}$. Then $x'y'\in E$, otherwise $\{y',x_4,y_4,x',x_3\}$ induces a bull. So the number of nonhomogeneous 2-chromatic components of $G[S_4(1)]$ is not more than 4, otherwise the vertices mixed on them respectively can induce a $K_5$. 
			
			So there are at most five 2-chromatic components in $G[S_4(1)]$. It follows that $S_3\cup S_4$ is bounded.
		\end{proof}
		
		\begin{claim}\label{S3(i) 2-chromatic}
			If $\chi(S_3(i))=2$ for some $1\le i\le 5$, then $S_3\cup S_4$ is bounded.
		\end{claim}
		
		\begin{proof}
			Without loss of generality, we assume $\chi(S_3(3))=2$. It follows from \ref{S3(i)S3(i+1)} that $S_3(2)=S_3(4)=\emptyset$, otherwise $S_3(3)\cup S_3(2)\cup \{v_2,v_3\}$ or $S_3(3)\cup S_3(4)\cup \{v_4,v_3\}$ contains an induced $K_5$. Similarly, it follows from \ref{S2(i)S4(i+2)} that $S_4(1)=S_4(5)=\emptyset$. Since $G$ has no subgraph isomorphic to $F_9$, $\chi(S_4(i))\le 1$ for each $2\le i\le 4$ and $\chi(S_3(j))\le 1$ for each $j=1,5$. By Claims \ref{S3 trivial}-\ref{S4 trivial}, $(\cup_{i=1,5}S_3(i))\cup S_4$ is bounded and the number of $K_1$-components in $G[S_3(3)]$ is also bounded.
			
			We now show that the number of vertices in a 2-chromatic component of $G[S_3(3)]$ is bounded. Let $A$ be a 2-chromatic component of $G[S_3(3)]$ and so $A$ is bipartite. Let the bipartition of $A$ be $(X,Y)$. Suppose that $|X|>2^{|S_2(3)\cup S_5\cup (\cup_{i=1,5}S_3(i))\cup (\cup_{2\le i\le 4}S_4(i))|}$. By the pigeonhole principle, there exists two vertices $x_1,x_2\in X$ which have the same neighborhood in $S_2(3)\cup S_5\cup (\cup_{i=1,5}S_3(i))\cup (\cup_{2\le i\le 4}S_4(i))$. Since $x_1$ and $x_2$ are not comparable, there must exist $y_1\in N(x_1)\setminus N(x_2),y_2\in N(x_2)\setminus N(x_1)$. Clearly, $y_1,y_2\in Y$ and so $\{x_1,x_2,y_1,y_2\}$ induces a $2K_2$ in $A$. Since $A$ is connected and bipartite, $A$ contains a $P_5$ by \autoref{2K2}, a contradiction. Thus $|X|\le 2^{|S_2(3)\cup S_5\cup (\cup_{i=1,5}S_3(i))\cup (\cup_{2\le i\le 4}S_4(i))|}$. Similarly, 
\[|Y|\le 2^{|S_2(3)\cup S_5\cup (\cup_{i=1,5}S_3(i))\cup (\cup_{2\le i\le 4}S_4(i))|}.\] 
Thus the number of vertices in $A$ is bounded.
			
			Then we show that there are at most $(2^{|S_2(3)|}+4)$ 2-chromatic components in $G[S_3(3)]$. 
			
			Suppose that the number of homogeneous 2-chromatic components of $G[S_3(3)]$ is more than $2^{|S_2(3)|}$. By the pigeonhole principle, there are two 2-chromatic components $A_1,A_2$ such that $N_{S_2(3)}(A_1)=N_{S_2(3)}(A_2)$. By \autoref{lem:XY}, $A_1$ is not complete to $N(A_2)$ and $A_2$ is not complete to $N(A_1)$. So there must exist $z_1\in N(A_1)\setminus N(A_2)$ and $z_2\in N(A_2)\setminus N(A_1)$. Clearly, $z_1,z_2\in (\cup_{i=1,5}S_3(i))\cup (\cup_{2\le i\le 4}S_4(i))\cup S_5$. Then $\{v_1,z_1,z_2\}\cup A_1\cup A_2$ contains an induced bull or an induced $P_5$, depending on whether $z_1z_2\in E$, a contradiction. Thus there are at most $2^{|S_2(3)|}$ homogeneous 2-chromatic components in $G[S_3(3)]$.
			
			Let $B_1,B_2$ be two nonhomogeneous 2-chromatic components of $G[S_3(3)]$. So there exists $x'$ mixed on $B_1$ and $y'$ mixed on $B_2$. Let $x'$ be mixed on edge $x_3y_3$ in $B_1$ and $y'$ be mixed on edge $x_4y_4$ in $B_2$. By symmetry, assume that $x'x_3,y'x_4\in E$ and $x'y_3,y'y_4\notin E$. It is evident that $x'$ and $y'$ are not the same vertex, otherwise $\{y_3,x_3,x',x_4,y_4\}$ induces a $P_5$. Similarly, $x'$ is not mixed on $x_4y_4$ and $y'$ is not mixed on $x_3y_3$. Clearly, $x',y'\in S_2(3)\cup S_5\cup (\cup_{i=1,5}S_3(i))\cup (\cup_{2\le i\le 4}S_4(i))$.
			
			\noindent{\bf Case 1.} $x'$ is anticomplete to $\{x_4,y_4\}$ and $y'$ is anticomplete to $\{x_3,y_3\}$. Then $x'$ is nonadjacent to $y'$, otherwise $\{y_3,x_3,x',y',x_4,y_4\}$ induces a $P_6$. If $x',y'\notin S_2(3)$, then $\{x_3,x',v_1,y',x_4\}$ induces a $P_5$. If $x',y'\in S_2(3)$, then $\{x',x_3,v_3,x_4,y'\}$ induces a $P_5$. So assume $x'\in S_2(3)$ and $y'\notin S_2(3)$. Then $\{x_4,v_3,y_3,x_3,x'\}$ induces a bull, a contradiction. 
			
			\noindent{\bf Case 2.} $x'$ is complete to $\{x_4,y_4\}$ and $y'$ is anticomplete to $\{x_3,y_3\}$. Then $x'y'\in E$, otherwise $\{y',x_4,y_4,x',x_3\}$ induces a bull. So as the case when $x'$ is anticomplete to $\{x_4,y_4\}$ and $y'$ is complete to $\{x_3,y_3\}$.
			
			\noindent{\bf Case 3.} $x'$ is complete to $\{x_4,y_4\}$ and $y'$ is complete to $\{x_3,y_3\}$. Suppose that $x',y'\notin S_2(3)$ and so $\{y_4,x',v_1,y',y_3\}$ induces a bull or a $P_5$, depending on whether $x'y'\in E$, a contradiction. If $x',y'\in S_2(3)$, then $x'y'\in E$, otherwise $\{x',y_4,v_3,y_3,y'\}$ induces a $P_5$. If $x'\in S_2(3)$ and $y'\notin S_2(3)$, then $x'y'\in E$, otherwise $\{v_1,y',y_3,x_3,x'\}$ induces a bull.
			
			We now know that $x'$ must be adjacent to $y'$. So the number of nonhomogeneous 2-chromatic components of $G[S_3(3)]$ is not more than 4, otherwise the vertices mixed on them respectively can induce a $K_5$, a contradiction. It follows that there are at most $(2^{|S_2(3)|}+4)$ 2-chromatic components in $G[S_3(3)]$. 
			
			Therefore, $S_3\cup S_4$ is bounded.
		\end{proof}
		
		By Claims \ref{S3 trivial}-\ref{S3(i) 2-chromatic}, $S_3\cup S_4$ is bounded and so is $|G|$. This completes the proof of \autoref{th}.
	\end{proof}

\end{document}